\documentclass[reqno,english]{amsart}%
\usepackage{amsfonts,amsmath,latexsym,verbatim,amscd,mathrsfs,color,array}
\usepackage[colorlinks=true]{hyperref}
\usepackage{amsmath,amssymb,amsthm,amsfonts,graphicx,color}
\usepackage{amssymb}
\usepackage{pdfsync}
\usepackage{epstopdf}
\usepackage{cite}
\usepackage{graphicx}
\usepackage{amsmath}
\usepackage{refcheck}
\usepackage{amsfonts}%
\providecommand{\U}[1]{\protect\rule{.1in}{.1in}}
\numberwithin{equation}{section}

\newtheorem{theorem}{Theorem}[section]
\newtheorem{corollary}{Corollary}[section]
\newtheorem{lemma}{Lemma}[section]
\newtheorem{proposition}{Proposition}[section]
\newtheorem{remark}{Remark}[section]
\newtheorem{definition}{Definition}[section]

\numberwithin{equation}{section}

\newcommand{\bbr}{\mathbb{R}}

\newcommand{\bbn}{\mathbb{N}}

\newcommand{\bd}{\begin{definition}}
\newcommand{\ed}{\end{definition}}

\newcommand{\br}{\begin{remark}}
\newcommand{\er}{\end{remark}}

\newcommand{\be}{\begin{equation}}
\newcommand{\ee}{\end{equation}}

\newcommand{\bc}{\begin{corollary}}
\newcommand{\ec}{\end{corollary}}
\begin{document}

\title[Stability of the logarithmic Sobolev inequality]{Sharp stability of the logarithmic Sobolev inequality in the critical point setting}

\author[J. Wei]{Juncheng Wei}
\address{\noindent Department of Mathematics, University of British Columbia,
Vancouver, B.C., Canada, V6T 1Z2}
\email{jcwei@math.ubc.ca}

\author[Y.Wu]{Yuanze Wu}
\address{\noindent  School of Mathematics, China
University of Mining and Technology, Xuzhou, 221116, P.R. China }
\email{wuyz850306@cumt.edu.cn}

\begin{abstract}
In this paper, we consider the Euclidean logarithmic Sobolev inequality
\begin{eqnarray*}
\int_{\bbr^d}|u|^2\log|u|dx\leq\frac{d}{4}\log\bigg(\frac{2}{\pi d e}\|\nabla u\|_{L^2(\bbr^d)}^2\bigg),
\end{eqnarray*}
where $u\in W^{1,2}(\bbr^d)$ with $d\geq2$ and $\|u\|_{L^2(\bbr^d)}=1$.  It is well known that extremal functions of this inequality are precisely the
Gaussians
\begin{eqnarray*}
\mathfrak{g}_{\sigma,z}(x)=(\pi\sigma)^{-\frac{d}{2}}\mathfrak{g}_{*}\bigg(\sqrt{\frac{\sigma}{2}}(x-z)\bigg)\quad\text{with}\quad \mathfrak{g}_{*}(x)=e^{-\frac{|x|^2}{2}}.
\end{eqnarray*}
We prove that if $u\geq0$ satisfying $(\nu-\frac12)c_0<\|u\|_{H^1(\bbr^d)}^2<(\nu+\frac12)c_0$ and $\|-\Delta u+u-2u\log |u|\|_{H^{-1}}\leq\delta$, where $c_0=\|\mathfrak{g}_{1,0}\|_{H^1(\bbr^d)}^2$, $\nu\in\bbn$ and $\delta>0$ sufficiently small, then
\begin{eqnarray*}
\text{dist}_{H^1}(u, \mathcal{M}^\nu)\lesssim\|-\Delta u+u-2u\log |u|\|_{H^{-1}}
\end{eqnarray*}
which is optimal in the sense that the order of the right hand side is sharp,
where
\begin{eqnarray*}
\mathcal{M}^\nu=\{(\mathfrak{g}_{1,0}(\cdot-z_1), \mathfrak{g}_{1,0}(\cdot-z_2), \cdots, \mathfrak{g}_{1,0}(\cdot-z_\nu))\mid z_i\in\bbr^d\}.
\end{eqnarray*}
Our result provides an optimal stability of the Euclidean logarithmic Sobolev inequality in the critical point setting.

\vspace{3mm} \noindent{\bf Keywords: Euclidean logarithmic Sobolev inequality; Optimal stability; Critical point setting.}

\vspace{3mm}\noindent {\bf AMS} Subject Classification 2010: 35A23; 35B35.%

\end{abstract}

\date{}
\maketitle

\section{Introduction}
A fundamental task in understanding functional inequalities, arise in the calculus of variations, geometry, etc, is to study the best constants, the classification of extremal functions, as well as their qualitative properties for parameters in the full region, since such functional inequalities are crucial in understanding nonlinear partial differential equations (Nonlinear PDEs for short) by virtue of the complete knowledge of the best constants, extremal functions, and qualitative properties.  The most well-studied functional inequality in the community of Nonlinear PDEs is the Sobolev inequality, whose classical one with exponent $2$ states that for any $u\in H^1(\bbr^d)$ with $d\geq3$, there holds
\begin{eqnarray}\label{eqnn0001}
S\bigg(\int_{\bbr^d}|u|^{\frac{2d}{d-2}}dx\bigg)^{\frac{d-2}{d}}\leq\int_{\bbr^d}|\nabla u|^2dx,
\end{eqnarray}
where $S>0$ is a constant which is only dependent of the dimension and $H^1(\bbr^d)$ is the classical Sobolev space given by
\begin{eqnarray*}
H^1(\bbr^d)=\{u\in L^{\frac{2d}{d-2}}(\bbr^d)\mid |\nabla u|\in L^2(\bbr^d)\}.
\end{eqnarray*}
It has been proved in \cite{A1976,T1976} that
\begin{eqnarray*}
S=\pi d(d-2)\bigg(\frac{\Gamma(\frac{d}{2})}{\Gamma(d)}\bigg)^{\frac{2}{d}}
\end{eqnarray*}
is optimal and the extremal functions, which are called the Aubin-Talenti bubbles in the literature, are given by
\begin{eqnarray*}
U_{\lambda,z,c}(x)=c\bigg(\frac{1}{1+\lambda^2|x-z|^2}\bigg)^{\frac{d-2}{2}},
\end{eqnarray*}
where $\lambda>0$, $c\in\bbr$ and $z\in\bbr^d$.

\vskip0.12in

Once a functional inequality is well understood for its best constant and extremal functions, it is natural to concern its stability, which is growingly interested in recent years by its important applications in understanding many Nonlinear PDEs, such as the fast diffusion equation, the Keller-Segel equation and so on.  The basic question one wants to address in this aspect is the following (cf. \cite{F2013}):
\begin{enumerate}
\item[$(Q)$]\quad Suppose we are given a functional inequality for which minimizers are known.
Can we prove, in some quantitative way, that if a function ``almost attains the
equality'' then it is close (in some suitable sense) to one of the minimizers?
\end{enumerate}
Such study was first raised by Brezis and Lieb in \cite{BL1985} for the classical Sobolev inequality~\eqref{eqnn0001} as an open question, which was settled by Bianchi and Egnell in \cite{BE1991} by proving that
\begin{eqnarray*}
dist_{H^{1}}^2(u, \mathcal{Z})\lesssim\|\nabla u\|^2_{L^2(\bbr^d)}-S\|u\|^2_{L^{\frac{2d}{d-2}}(\bbr^N)},
\end{eqnarray*}
where $\|\cdot\|_{L^p(\bbr^N)}$ is the usual norm in the Lebesgue space $L^p(\bbr^d)$ and
\begin{eqnarray*}
\mathcal{Z}=\{U_{\lambda,z,c}\mid (\lambda,z,c)\in\bbr_+\times\bbr^{d+1}\}.
\end{eqnarray*}

\vskip0.12in

Unlike Bianchi and Egnell's study (\cite{BE1991}) in the functional inequality setting, in recent years, Figalli et al. initiated the study on the stability of the classical Sobolev inequality~\eqref{eqnn0001} in the critical point setting, that is, studying the stability of the Euler-Lagrange equation of the classical Sobolev inequality~\eqref{eqnn0001}.  The study on the stability of the classical Sobolev inequality~\eqref{eqnn0001} in the critical point setting is more challenging since the Euler-Lagrange equation of the classical Sobolev inequality~\eqref{eqnn0001} has sign-changing solutions.  Moreover, the ``almost'' solutions of the Euler-Lagrange equation of the classical Sobolev inequality~\eqref{eqnn0001} may decompose into several parts at infinity (cf. \cite{S1984}).  By setting the study in a suitable way, Figalli et al. proved that
\begin{enumerate}
\item[$(1)$] (Ciraolo-Figalli-Maggi \cite{CFM2017})\quad Let $d\geq3$ and $u\in D^{1,2}(\bbr^d)$ be positive such that $\|\nabla u\|_{L^2(\bbr^d)}^2\leq\frac{3}{2}S^{\frac{d}{2}}$ and $\|\Delta u+|u|^{\frac{4}{d-2}}u\|_{H^{-1}}\leq\delta$ for some $\delta>0$ sufficiently small.  Then,
    \begin{eqnarray*}
    dist_{D^{1,2}}(u, \mathcal{M}_0)\lesssim\|\Delta u+|u|^{\frac{4}{d-2}}u\|_{H^{-1}},
    \end{eqnarray*}
    where $\mathcal{M}_0=\{(U[z, \lambda]\mid z\in\bbr^d, \lambda>0\}.$
\item[$(2)$] (Figalli-Glaudo \cite{FG2021})\quad Let $u\in D^{1,2}(\bbr^d)$ be nonnegative such that
\begin{eqnarray*}
(\nu-\frac12)S^\frac{d}{2}<\|u\|_{D^{1,2}(\bbr^N)}^2<(\nu+\frac12)S^\frac{d}{2}
\end{eqnarray*}
and $\|\Delta u+|u|^{\frac{4}{d-2}}u\|_{H^{-1}}\leq\delta$ for some $\delta>0$ sufficiently small.
Then, for $3\leq d\leq5$,
    \begin{eqnarray*}
    dist_{D^{1,2}}(u, \mathcal{M}_0^\nu)\lesssim\|\Delta u+|u|^{\frac{4}{d-2}}u\|_{H^{-1}}
    \end{eqnarray*}
    where
    \begin{eqnarray*}
    \mathcal{M}_0^\nu=\{(U[z_1, \lambda_1],U[z_2, \lambda_2],\cdots,U[z_\nu, \lambda_\nu])\mid z_i\in\bbr^d, \lambda_i>0\}.
    \end{eqnarray*}
\end{enumerate}
All the above results are optimal in the sense that the orders of the right hand sides in the above estimates are sharp, while the optimal stability of the classical Sobolev inequality~\eqref{eqnn0001} in the critical point setting for the case $N\geq6$ was left in \cite{FG2021} as an open problem which was solved by Deng, Sun and Wei in \cite{DSW2021}, very recently, by proving the following optimal stability:
\begin{eqnarray*}
dist_{D^{1,2}}(u, \mathcal{M}_0^\nu)\lesssim\left\{\aligned&\|\Delta u+|u|u\|_{H^{-1}}|\ln(\|\Delta u+|u|u\|_{H^{-1}})|^{\frac12},\ \ d=6;\\
&\|\Delta u+|u|^{\frac{4}{d-2}}u\|_{H^{-1}}^{\frac{d+2}{2(d-2)}},\ \ d\geq7.
\endaligned\right.
\end{eqnarray*}

\vskip0.12in

According to the important applications in understanding many Nonlinear PDEs, serval other famous functional inequalities, such as the Gagliardo-Nirenberg-Sobolev inequality (cf. \cite{CF2013,CFL2014,DT2013,DT16,N2019,R2014,S2019}), the Hardy-Littlewood-Sobolev inequality (cf. \cite{BT2002,DD2004,FT2002,GGMT1976,H1997,L1983,RSW2002}), the Caffarelli-Kohn-Nirenberg inequality (cf. \cite{CKN1984,CW2001,CC1993,DELT2009,DEL2012,DEL2016,FS2003,LW2004,WW2000,WW2003,WW22}), the $L^p$-Sobolev inequality (cf. \cite{C2006,CFMP2009,FN2019,FZ2020,F2015,FMP2007,N2020}) and so on, are also wildly studied on the best constants, the classification of extremal functions and stability.  However, according to their multi-parameters or no-Hilbert, most of these functional inequalities are far from well understood except for some special cases.

\vskip0.12in

It is worth pointing out that, besides the classical Sobolev inequality~\eqref{eqnn0001}, the Euclidean logarithmic Sobolev inequality,
\begin{eqnarray}\label{eq0001}
\int_{\bbr^d}|u|^2\log|u|dx\leq\frac{d}{4}\log\bigg(\frac{2}{\pi d e}\|\nabla u\|_{L^2(\bbr^d)}^2\bigg)
\end{eqnarray}
where $u\in W^{1,2}(\bbr^d)$ with $d\geq2$ and $\|u\|_{L^2(\bbr^d)}=1$, is also well understood for the best constants and the classification of extremal functions.  It is well known (cf. \cite{dPD03,W78}) that this inequality is optimal and extremal functions of \eqref{eq0001} are precisely the
Gaussians
\begin{eqnarray}\label{eq0002}
\mathfrak{g}_{\sigma,z}(x)=(\pi\sigma)^{-\frac{d}{2}}\mathfrak{g}_{*}\bigg(\sqrt{\frac{\sigma}{2}}(x-z)\bigg)\quad\text{with}\quad \mathfrak{g}_{*}(x)=e^{-\frac{|x|^2}{2}}.
\end{eqnarray}
Moreover, it is equivalent to the Gross logarithmic inequality (cf. \cite{G75}) with respect to Gaussian weight
\begin{eqnarray*}
\int_{\bbr^d}|g|^2\log|g|d\mu\leq\int_{\bbr^d}|\nabla g|^2d\mu\quad\text{with}\int_{\bbr^d}|g|^2d\mu=1\text{ and }d\mu=(2\pi)^{-\frac{d}{2}}e^{-\frac{|x|^2}{2}}.
\end{eqnarray*}
The Euclidean logarithmic Sobolev inequality has another two different forms:
\begin{eqnarray}\label{eq0005}
\int_{\bbr^d}|v|^2\log|v|^2dx\leq\frac{a^2}{\pi}\|\nabla v\|_{L^2(\bbr^d)}^2+\bigg(\log\|v\|_{L^2(\bbr^d)}^2-d(1+\log a)\bigg)\|v\|_{L^2(\bbr^d)}^2
\end{eqnarray}
for all $a>0$ and any function $v\in W^{1,2}(\bbr^d)$, and
\begin{eqnarray}\label{eq0006}
\int_{\bbr^d}|w|^2\log|w|^2dx+\frac{d}{2}(1+\frac12\log(2\pi))\leq\|\nabla w\|_{L^2(\bbr^d)}^2
\end{eqnarray}
for any function $w\in W^{1,2}(\bbr^d)$, $d\geq2$, with $\|w\|_{L^2(\bbr^d)}=1$.  \eqref{eq0005} and \eqref{eq0006} are equivalent by the relation $w(x)=v(\frac{a}{\sqrt{2\pi}}x)$ and they are established in \cite[Theorem~8.14]{LL01} and \cite[Theorem~1.2]{WZ19}, respectively.  However, \eqref{eq0001} is optimal since it can be obtained by minimizing the right hand side of \eqref{eq0005} for $a>0$ (cf. \cite{DT16}).

\vskip0.12in

The stability of the Euclidean logarithmic Sobolev inequality~\eqref{eq0001} in the functional inequality setting has also been widely studied, cf. \cite{BGRS2014,C1991,CF2013,DT16,FIL2016,FIPR2017,IM2014,IK2021,W78} and the references therein.  It is worth pointing out that most of these results are devoted to more general version of the logarithmic Sobolev inequality in the probability setting which contains the Euclidean logarithmic Sobolev inequality~\eqref{eq0001} as a special case for the Gaussians measure $d\mu=(2\pi)^{-\frac{d}{2}}e^{-\frac{|x|^2}{2}}$.
Thus, it is natural to consider the stability of the Euclidean logarithmic Sobolev inequality~\eqref{eq0001} in the critical point setting, as that for the classical Sobolev inequality~\eqref{eqnn0001}.  Let
\begin{eqnarray*}
\mathcal{E}(u)=\frac{d}{4}\log\bigg(\frac{2}{\pi d e}\|\nabla u\|_{L^2(\bbr^d)}^2\bigg)-\int_{\bbr^d}|u|^2\log|u|dx.
\end{eqnarray*}
Then critical points of $\mathcal{E}(u)$ in $H^1(\bbr^d)$ with the finite energy on the smooth manifold
\begin{eqnarray*}
\mathcal{S}=\{u\in H^1(\bbr^d)\mid \|u\|_{L^2(\bbr^d)}=1\}
\end{eqnarray*}
satisfy $\int_{\bbr^d}|u|^2|\log|u|^2|dx<+\infty$ and the following Euler-Lagrange equation
\begin{eqnarray}\label{eq0007}
-(\frac{d}{2\|\nabla u\|_{L^2(\bbr^d)}^2})\Delta u-(1+2\sigma_u)u=2u\log|u|\quad\text{in }\bbr^d,
\end{eqnarray}
where $\sigma_u$ is a part of unknown and appears in \eqref{eq0007} as the Lagrange multiplier.  Thus, \eqref{eq0007} is the Euler-Lagrange equation of \eqref{eq0001}.
Let
\begin{eqnarray*}
u_\lambda(x)=\lambda^\frac{d}{2}u(\lambda x)\quad\text{where }\lambda=\bigg(\frac{d}{2\|\nabla u\|_{L^2(\bbr^d)}^2}\bigg)^{\frac12},
\end{eqnarray*}
then by \eqref{eq0007} and a direct calculation, $u_{\lambda}$ satisfies
\begin{eqnarray}\label{eq0008}
-\Delta u_\lambda+(d\log\lambda-1-2\sigma_u)u_\lambda=2u_\lambda\log |u_\lambda|\quad\text{in }\bbr^d.
\end{eqnarray}
Let
\begin{eqnarray*}
u_\lambda^*(x)=\alpha_u u_\lambda(x)\quad\text{where }\alpha_u=e^{\frac{d}{2}\log\lambda-\sigma_u-1},
\end{eqnarray*}
then by \eqref{eq0008} and a direct calculation, $u_{\lambda}^*$ satisfies the equation
\begin{eqnarray}\label{eq0009}
-\Delta u+u=2u\log |u|\quad\text{in }\bbr^d.
\end{eqnarray}
Thus, the logarithmic Schr\"odinger equation~\eqref{eq0009} can be seen as the Euler-Lagrange equation of the Euclidean logarithmic Sobolev inequality~\eqref{eq0001}.

\vskip0.12in

It has been proved in \cite{dAMS14,T16} that the Gaussian
\begin{eqnarray}\label{eq0011}
\mathfrak{g}=e^{\frac{1+d}{2}}\mathfrak{g}_{*}
\end{eqnarray}
is the unique positive solution of the logarithmic Schr\"odinger equation~\eqref{eq0009} which satisfies $u(x)\to0$ as $|x|\to+\infty$, where $\mathfrak{g}_{*}$ is given by \eqref{eq0002}.  We remark that any solution of \eqref{eq0009} satisfying $\int_{\bbr^d}|u|^2|\log|u|^2|dx<+\infty$ must exponentially decay to zero as $|x|\to+\infty$ by the standard applications of the maximum principle.
Moreover, it has been proved in \cite[Theorem~1.3]{dAMS14} (see also \cite[Theorem~7.5]{GS11}) that $\mathfrak{g}$ is nondegenerate in the sense that $Ker(\mathcal{L})=span\{\partial_{x_j}\mathfrak{g}\}$ where
\begin{eqnarray}\label{eq0043}
\mathcal{L}=-\Delta+|x|^2-(d+2)
\end{eqnarray}
is the linearized operator of \eqref{eq0009} at $\mathfrak{g}$.  Note that \eqref{eq0009} is invariant under translations, thus, the smooth manifold
\begin{eqnarray*}
\mathcal{M}=\{\mathfrak{g}(\cdot-z)\mid z\in\bbr^d\}
\end{eqnarray*}
contains all positive solutions of \eqref{eq0009} satisfying $\int_{\bbr^d}|u|^2|\log|u|^2|dx<+\infty$.

\vskip0.12in

Let $\{u_n\}$, $u_n\geq0$ for all $n$, be bounded in $H^1(\bbr^d)$ and almost solves \eqref{eq0009}, that is, $\|-\Delta u_n+u_n-2u_n\log |u_n|\|_{H^{-1}}\to0$ as $n\to\infty$.  Then it is easy to see that $\{|u_n|^2|\log|u_n|^2|\}$ is bounded in $L^1(\bbr^d)$.
Thanks to the Brezis-Lieb lemma \cite[Lemma~3.1]{S19} and the positivity of the energy of the unique positive solution of \eqref{eq0009} (cf. \cite[Lemma~3.3]{S19}), it is standard (cf. \cite[Proposition~3.1]{AJ22}) to prove the following Struwe's decomposition of $\{u_n\}$.
\begin{proposition}\label{prop0001}
There exists $\nu\in\bbn$ and $y_{1,n},y_{2,n},\cdots y_{\nu,n}\subset\bbr^d$ such that
\begin{eqnarray*}
\|u_n-\sum_{i=1}^{\nu}\mathfrak{g}(\cdot-y_{i,n})\|_{H^1(\bbr^d)}\to0\quad\text{as }n\to\infty.
\end{eqnarray*}
\end{proposition}
By Proposition~\ref{prop0001}, for nonnegative functions $u_n$ with uniform $H^1$-bound, we know that there exists $\nu\in\bbn$ such that $\{u_n\}$ will be close to the manifold
\begin{eqnarray*}
\mathcal{M}^\nu=\{(\mathfrak{g}(\cdot-z_1), \mathfrak{g}(\cdot-z_2), \cdots, \mathfrak{g}(\cdot-z_\nu))\mid z_i\in\bbr^d\}
\end{eqnarray*}
in the $H^1$-topology.  Thus, it is natural to ask, if $u\geq0$ satisfy
\begin{eqnarray*}
(\nu-\frac12)c_0<\|u\|_{H^1(\bbr^d)}^2<(\nu+\frac12)c_0\quad\text{and}\quad \|-\Delta u+u-2u\log |u|\|_{H^{-1}}\leq\delta,
\end{eqnarray*}
where $c_0=\|\mathfrak{g}\|_{H^1(\bbr^d)}^2$, $\nu\in\bbn$ and $\delta>0$ sufficiently small, can we obtain a quantitative version of Proposition~\ref{prop0001} by optimally controlling $dist_{H^1}(u, \mathcal{M}^\nu)$ by $\|-\Delta u+u-2u\log |u|\|_{H^{-1}}$?  The purpose of this paper is to give a positive answer to this natural question and our main result reads as follows.
\begin{theorem}\label{thmnn0001}
Let $u\geq0$ satisfy
\begin{eqnarray*}
(\nu-\frac12)c_0<\|u\|_{H^1(\bbr^d)}^2<(\nu+\frac12)c_0\quad\text{and}\quad \|-\Delta u+u-2u\log |u|\|_{H^{-1}}\leq\delta,
\end{eqnarray*}
where $c_0=\|\mathfrak{g}\|_{H^1(\bbr^d)}^2$, $\nu\in\bbn$ and $\delta>0$ sufficiently small.  Then
\begin{eqnarray}\label{eqnn0002}
\text{dist}_{H^1}(u, \mathcal{M}^\nu)\lesssim\|-\Delta u+u-2u\log |u|\|_{H^{-1}},
\end{eqnarray}
where
\begin{eqnarray*}
\mathcal{M}^\nu=\{(\mathfrak{g}_{1,0}(\cdot-z_1), \mathfrak{g}_{1,0}(\cdot-z_2), \cdots, \mathfrak{g}_{1,0}(\cdot-z_\nu))\mid z_i\in\bbr^d\}.
\end{eqnarray*}
Moreover, \eqref{eqnn0001} is optimal in the sense that the order of the right hand side is sharp.
\end{theorem}

\begin{remark}
The main ideas in proving Theorem~\ref{thmnn0001} is similar to that of \cite{DSW2021,WW22}, that is, we choose special $y_{1,\delta},y_{2,\delta},\cdots y_{\nu,\delta}\subset\bbr^d$ and decompose $u-\sum_{i=1}^{\nu}\mathfrak{g}(\cdot-y_{i,\delta})$ into two parts, where the first part is very regular which can be well estimated and the second part is much smaller than the first part.  However, the logarithmic nonlinearity is very different from the power-type one dealt with in \cite{DSW2021,WW22}, thus, we need to careful estimate it to make sure that the strategy in \cite{DSW2021,WW22} work for \eqref{eq0009}.
\end{remark}

\vskip0.12in

\noindent{\bf\large Notations.} Throughout this paper, $C$ and $C'$ are indiscriminately used to denote various absolutely positive constants.  $\sigma,\sigma',\sigma''$ are indiscriminately used to denote various absolutely positive constants which can be taken arbitrary small.  $a\sim b$ means that $C'b\leq a\leq Cb$ and $a\lesssim b$ means that $a\leq Cb$.

\section{Preliminaries}
Let us consider the equation
\begin{eqnarray}\label{eq0004}
-\Delta u+u=2u\log |u|+f,\quad\text{in } \bbr^d,
\end{eqnarray}
where $f\in H^{-1}$ and $u\in H^1(\bbr^d)$ is nonnegative and satisfies
\begin{eqnarray*}
(\nu-\frac12)c_0<\|u\|_{H^1(\bbr^d)}^2<(\nu+\frac12)c_0\quad\text{for some fixed }\nu\in\bbn
\end{eqnarray*}
with $c_0=\|\mathfrak{g}\|_{H^1(\bbr^d)}^2$.  By multiplying \eqref{eq0004} with $u$ and integrating by parts,
it is easy to see that $\int_{\bbr^d}|u|^2|\log|u|^2|dx\lesssim1$.  Moreover,
by Proposition~\ref{prop0001}, there exists $z_{1,f},z_{2,f},\cdots z_{\nu,f}\subset\bbr^d$ such that
\begin{eqnarray*}
\|u-\sum_{i=1}^{\nu}\mathfrak{g}(\cdot-z_{i,f})\|_{H^1(\bbr^d)}\to0\quad\text{as }\|f\|_{H^{-1}}\to0.
\end{eqnarray*}
It follows that
\begin{eqnarray}\label{eq0003}
\mathfrak{c}_f=\inf_{z_i\in\bbr^d}\|u-\sum_{i=1}^\nu\mathfrak{g}(\cdot-z_i)\|_{H^1(\bbr^d)}^2\to0\quad\text{as }\|f\|_{H^{-1}}\to0.
\end{eqnarray}
Thus, by solving the minimizing problem~\eqref{eq0003} in a standard way (cf. \cite{BC88}),
we can write $u=\sum_{i=1}^\nu\mathfrak{g}(\cdot-y_{i,f})+\rho_{f}$, where $\{y_{i,f}\}$ is the solution of \eqref{eq0003} and the remaining term $\rho_f$ satisfies
\begin{eqnarray*}\label{eq0039}
\|\rho_f\|_{H^1(\bbr^d)}^2=\mathfrak{c}_f\to0\quad\text{as }\|f\|_{H^{-1}}\to0
\end{eqnarray*}
and
the orthogonal conditions
\begin{eqnarray}\label{eq0044}
\langle \rho_f, \partial_{x_l}\mathfrak{g}_{j,f}\rangle_{H^1(\bbr^d)}=0,\quad l=1,2,\cdots,d\text{ and }j=1,2,\cdots,\nu.
\end{eqnarray}

\vskip0.12in

For the sake of simplicity, we denote $\mathfrak{g}(\cdot-y_{i,f})$ by $\mathfrak{g}_{i,f}$.  Clearly, by \eqref{eq0004}, $\rho_f$ satisfies
\begin{eqnarray}\label{eq0033}
\left\{\aligned&-\Delta\rho_f+\rho_f=2(\sum_{i=1}^\nu\mathfrak{g}_{i,f}+\rho_{f})\log(\sum_{i=1}^\nu\mathfrak{g}_{i,f}+\rho_{f})
-2\sum_{i=1}^\nu\mathfrak{g}_{i,f}\log\mathfrak{g}_{i,f}
+f,\quad\text{in }H^{-1},\\
&\langle \rho_f, \partial_{x_l}\mathfrak{g}_{i,f}\rangle_{H^1(\bbr^d)}=0,\quad l=1,2,\cdots,d\text{ and }i=1,2,\cdots,\nu.
\endaligned\right.
\end{eqnarray}
It is convenient to write \eqref{eq0033} as follows:
\begin{eqnarray}\label{eq0034}
\left\{\aligned&\mathcal{L}_f(\rho_f)=E+N(\rho_f)+f,\quad\text{in }H^{-1},\\
&\langle \rho_f, \partial_{x_l}\mathfrak{g}_{i,f}\rangle_{H^1(\bbr^d)}=0,\quad l=1,2,\cdots,d\text{ and }i=1,2,\cdots,\nu,
\endaligned\right.
\end{eqnarray}
where
\begin{eqnarray}\label{eq0019}
\mathcal{L}_f=-\Delta-1-2\log(\sum_{i=1}^\nu\mathfrak{g}_{i,f})
\end{eqnarray}
is the linear operator,
\begin{eqnarray}\label{eq0012}
E=2(\sum_{i=1}^\nu\mathfrak{g}_{i,f})\log(\sum_{i=1}^\nu\mathfrak{g}_{i,f})-2\sum_{i=1}^\nu\mathfrak{g}_{i,f}\log\mathfrak{g}_{i,f}
\end{eqnarray}
is the error and
\begin{eqnarray}
N(\rho_f)&=&2(\sum_{i=1}^\nu\mathfrak{g}_{i,f}+\rho_{f})\log(\sum_{i=1}^\nu\mathfrak{g}_{i,f}+\rho_{f})
-2(\sum_{i=1}^\nu\mathfrak{g}_{i,f})\log(\sum_{i=1}^\nu\mathfrak{g}_{i,f})\notag\\
&&-2(1+\log(\sum_{i=1}^\nu\mathfrak{g}_{i,f}))\rho_f\label{eq0031}
\end{eqnarray}
is the nonlinear part.

\section{Proof of (\ref{eqnn0002})}
Let
\begin{eqnarray}\label{eqnn0010}
\eta_{i}=\min_{j\not=i}\eta_{i,j}\quad\text{and}\quad\eta=\min_{i}\eta_{i},
\end{eqnarray}
where $\eta_{i,j}=|y_{j,f}-y_{i,f}|$.
Then by \eqref{eq0004}, \eqref{eq0003} and the uniqueness of $\mathfrak{g}$, we can easy to see that
\begin{eqnarray}\label{eqnn0030}
\eta=\min_{i\not=j}\{|y_{i,f}-y_{j,f}|\}\to+\infty\quad\text{as }\delta\to0,
\end{eqnarray}
where we denote $\|f\|_{H^{-1}}$ by $\delta$ for the sake of simplicity.
Let
\begin{eqnarray}\label{eq0045}
\Omega_{i}=\{x\in \bbr^d\mid \mathfrak{g}_{i,f}\geq\mathfrak{g}_{j,f}\quad\text{for all }j\not=i\}.
\end{eqnarray}
Then $\bbr^d=\cup_{i=1}^{\nu}\Omega_{i}$ and $\sum_{j=1}^\nu\mathfrak{g}_{j,f}\sim\mathfrak{g}_{i,f}$ in $\Omega_{i}$.  Moreover, we also introduce
\begin{eqnarray}\label{eq0075}
\Pi_{c,i,j}: 2(y_{i,f}-y_{j,f})x+c=0\quad\text{and}\quad \mathbb{L}_{i,j}:(x-y_{i,f})\times(y_{i,f}-y_{j,f})=0,
\end{eqnarray}
where $c\in\bbr$ are constants.  We take $x_{c,i,j}\in \Pi_{c.,i,j}\cap\mathbb{L}_{i,j}\cap\Omega_i$ and denote
\begin{eqnarray}\label{eqnn0003}
|x_{c,i,j}-y_{i,f}|=\pm\alpha_{c,i,j}\eta_{i,j},
\end{eqnarray}
where $\alpha_{c,i,j}>-\frac12$ with
\begin{eqnarray*}
\alpha_{c,i,j}\left\{\aligned&>0,\quad \langle x_{c,i,j}-y_{i,f}, y_{i,f}-y_{j,f}\rangle>0,\\
&<0,\quad \langle x_{c,i,j}-y_{i,f}, y_{i,f}-y_{j,f}\rangle<0.\endaligned\right.
\end{eqnarray*}
Then
\begin{eqnarray*}
|x_{c,i,j}-y_{j,f}|=(1+\alpha_{c,i,j})\eta_{i,j}.
\end{eqnarray*}
\begin{lemma}\label{lem0001}
We have,
\begin{eqnarray}
E&\sim&\sum_{i=1}^\nu\chi_{\Omega_i}\mathfrak{g}_{i,f}(\sum_{j\not=i}\chi_{\Pi_{c,i,j}}e^{-(\alpha_{c,i,j}+\frac12)\eta_{i,j}^2}\log(1+e^{(\alpha_{c,i,j}+\frac12)\eta_{i,j}^2}))\label{eq0017}
\end{eqnarray}
as $\|f\|_{H^{-1}}\to0$.
\end{lemma}
\begin{proof}
By \eqref{eq0012}, we have
\begin{eqnarray}\label{eq0013}
E=2\sum_{i=1}^\nu\mathfrak{g}_{i,f}\log(1+\sum_{j\not=i}\frac{\mathfrak{g}_{j,f}}{\mathfrak{g}_{i,f}}).
\end{eqnarray}
By \eqref{eq0013}, we write
\begin{eqnarray*}
E&=&2\mathfrak{g}_{i,f}\log(1+\sum_{l\not=i}\frac{\mathfrak{g}_{l,f}}{\mathfrak{g}_{i,f}})
+2\sum_{j\not=i}\mathfrak{g}_{j,f}\log(1+\sum_{l\not=j}\frac{\mathfrak{g}_{l,f}}{\mathfrak{g}_{j,f}})\\
&=:&I+II
\end{eqnarray*}
in $\Omega_{i}$, where $\Omega_i$ is given by \eqref{eq0045}.  For $I$, we have
\begin{eqnarray*}
I\sim\sum_{j\not=i}\mathfrak{g}_{j,f}\quad\text{in }\Omega_i.
\end{eqnarray*}
For $\mathfrak{g}_{j,f}\log(1+\sum_{l\not=j}\frac{\mathfrak{g}_{l,f}}{\mathfrak{g}_{j,f}})$ with $j\not=i$, we write
\begin{eqnarray*}
\mathfrak{g}_{j,f}\log(1+\sum_{l\not=j}\frac{\mathfrak{g}_{l,f}}{\mathfrak{g}_{j,f}})=\mathfrak{g}_{j,f}
\log(1+\frac{\mathfrak{g}_{i,f}+\sum_{l\not=j,i}\mathfrak{g}_{l,f}}{\mathfrak{g}_{j,f}}).
\end{eqnarray*}
Then in $\Omega_i$, we have
\begin{eqnarray}\label{eq0032}
E\sim II\sim\sum_{j\not=i}\mathfrak{g}_{j,f}\log(1+\frac{\mathfrak{g}_{i,f}}{\mathfrak{g}_{j,f}}).
\end{eqnarray}
Let us consider the function
\begin{eqnarray*}
\frac{\mathfrak{g}_{i,f}}{\mathfrak{g}_{j,f}},\quad \forall j\not=i.
\end{eqnarray*}
By \eqref{eq0002}, \eqref{eq0011} and direct calculations, we have
\begin{eqnarray}\label{eq0030}
\nabla(\frac{\mathfrak{g}_{i,f}}{\mathfrak{g}_{j,f}})=(y_{i,f}-y_{j,f})e^{-\frac12(|x-y_{i,f}|^2-|x-y_{j,f}|^2)}.
\end{eqnarray}
Then by \eqref{eq0030},
\begin{eqnarray*}
\frac{\mathfrak{g}_{i,f}}{\mathfrak{g}_{j,f}}\equiv const.\quad\text{on the hyperplane $\Pi_{c,i,j}$ for all $c\in\bbr$},
\end{eqnarray*}
where $\Pi_{c,i,j}$ is given by \eqref{eq0075}.
Thus, we have
\begin{eqnarray}\label{eq0016}
\frac{\mathfrak{g}_{i,f}}{\mathfrak{g}_{j,f}}\equiv e^{(\alpha_{c,i,j}+\frac12)|y_{i,f}-y_{j,f}|^2},\quad \forall x\in\Pi_{c,i,j}\cap\Omega_i,
\end{eqnarray}
where $\alpha_{c,i,j}>-\frac12$ is given by \eqref{eqnn0003}.  It follows from \eqref{eq0032}, \eqref{eq0016} and the fact that $\Pi_{c,i,j}\perp \mathbb{L}_{i,j}$ for all $c\in\bbr$ that \eqref{eq0017} holds.
\end{proof}

\vskip0.12in

As that in \cite{DSW2021,WW22}, we decompose $\rho_f=\phi_f+\varphi_f$, where $\phi_f$ is the solution of the following equation:
\begin{eqnarray}\label{eqnn0034}
\left\{\aligned&\mathcal{L}_f(\phi_f)=E+N(\phi_{f})-\sum_{j=1}^d\sum_{i=1}^{\nu}a_{j,i}\partial_{x_j}\mathfrak{g}_{i,f},\quad\text{in }H^{-1},\\
&\langle \phi_f, \partial_{x_j}\mathfrak{g}\rangle_{H^1(\bbr^d)}=0,\quad l=1,2,\cdots,d
\endaligned\right.
\end{eqnarray}
with $a_{j,i}$ being the Lagrange multipliers given by
\begin{eqnarray*}
a_{j,i}\sim\langle E+N(\phi_{f}), \partial_{x_j}\mathfrak{g}_{i,f}\rangle_{L^2}.
\end{eqnarray*}

\vskip0.12in

To solve \eqref{eqnn0034}, let us first establish a good linear theory.  Let
\begin{eqnarray*}
\mathfrak{g}_{i,f,d-1,j}=e^{\frac{1+d}{2}-\frac{|z^{i,j}|^2}{2}}
\end{eqnarray*}
and
\begin{eqnarray*}
D_{R,i,j}=\{x\in\Pi_{c,i,j}\mid |x_{c,i,j}-\frac{y_{i,f}+y_{j,f}}{2}|\leq R\},
\end{eqnarray*}
where $z^{i,j}\perp \mathbb{L}_{i,j}$ and $x_{c,i,j}$ is given by \eqref{eqnn0003}.  We define
\begin{eqnarray*}
\Omega_{i,j,d-1}=\{x\in\Omega_i\mid\mathfrak{g}_{i,f,d-1,j}\geq\mathfrak{g}_{i,f,d-1,l}, \forall l\not=i,j\}.
\end{eqnarray*}
Then as above, $\Omega_i=\cup_{j\not=i}\Omega_{i,j,d-1}$ and $\sum_{l\not=i}\mathfrak{g}_{i,f,d-1,l}\sim\mathfrak{g}_{i,f,d-1,j}$ in $\Omega_{i,j,d-1}$ for all $j$, where $\Omega_i$ is given by \eqref{eq0045}.
We introduce the norms
\begin{eqnarray*}
\|u\|_{\natural}&=&\sum_{i=1}^{\nu}\sum_{j\not=i}(\sup_{D_{\sigma\eta,i,j}^c\cap\Omega_{i,j,d-1}}\frac{|u|}{\mathfrak{g}_{i,f}^{1-\sigma}}
+\sup_{D_{\sigma\eta,i,j}\cap\Omega_{i,j,d-1}}\frac{|u|}{e^{-\frac{1}{8}(\eta_{i,j}^2-\eta^2)}\mathfrak{g}_{i,f,d-1,j}})
\end{eqnarray*}
and
\begin{eqnarray*}
\|u\|_{\sharp}&=&\sum_{i=1}^{\nu}\sum_{j\not=i}(\sup_{D_{\sigma\eta-1,i,j}^c\cap\Omega_{i,j,d-1}}\frac{|u|}{\mathfrak{g}_{i,f}^{1-\sigma}}
+\sup_{D_{\sigma\eta-1,i,j}\cap\Omega_{i,j,d-1}}\frac{\eta^{2}|u|}{e^{-\frac{1}{8}(\eta_{i,j}^2-\eta^2)}\mathfrak{g}_{i,f,d-1,j}}).
\end{eqnarray*}
Then
\begin{eqnarray*}
\mathbb{X}=\{u\in H^1(\bbr^d)\mid \|u\|_{\sharp}<+\infty\}\quad\text{and}\quad\mathbb{Y}=\{u\in L^2(\bbr^d)\mid \|u\|_{\natural}<+\infty\},
\end{eqnarray*}
are Banach spaces.

\vskip0.12in

Let us consider the following linear equation:
\begin{eqnarray}\label{eqnn0018}
\left\{\aligned&\mathcal{L}_f(\psi)=h-\sum_{i=1}^{\nu}\sum_{j=1}^d b_{j,i}\partial_{x_j}\mathfrak{g}_{i,f}
,\quad\text{in }\bbr^d,\\
&\phi\in\mathbb{X}^{\bot},
\endaligned\right.
\end{eqnarray}
where $\mathcal{L}_f$ is the linear operator given by \eqref{eq0019},
\begin{eqnarray}\label{eq0029}
\mathbb{X}^{\bot}=\{u\in\mathbb{X}\mid \langle u, \partial_{x_j}\mathfrak{g}_{i,f}\rangle_{H^1(\bbr^d)}=0,\quad j=1,2,\cdots,d; \quad i=1,2,\cdots,\nu\}
\end{eqnarray}
and $b_{j,i}$ are the Lagrange multipliers given by
\begin{eqnarray}\label{eq0028}
b_{j,i}\sim\langle h, \partial_{x_j}\mathfrak{g}_{i,f}\rangle_{L^2}.
\end{eqnarray}
\begin{lemma}\label{lem0004}
As $\|f\|_{H^{-1}}\to0$, \eqref{eqnn0018} is unique solvable for every $h\in\mathbb{Y}$ with $\|\psi\|_{\sharp}+\sum_{i=1}^{\nu}\sum_{j=1}^d |b_{j,i}|\lesssim\|h\|_{\natural}$.
\end{lemma}
\begin{proof}
By \eqref{eq0002} and \eqref{eq0011}, it is easy to see that for $R>0$ sufficiently large,
\begin{eqnarray*}
-(1+2\log(\sum_{j=1}^{\nu}\mathfrak{g}_{j,f}))\geq0
\end{eqnarray*}
in $(\cup_{j=1}^{\nu}B_{R}(y_{j,f}))^c$.
It follows from $\sum_{j=1}^{\nu}\mathfrak{g}_{j,f}\sim\mathfrak{g}_{i,f}$ in $\Omega_i$ that
\begin{eqnarray}\label{eq0025}
&-(1+2\log(\sum_{j=1}^{\nu}\mathfrak{g}_{j,f}))=|x-y_{i,f}|^2+O(1)\quad\text{in }D_{\sigma\eta+2,i,j}\cap\Omega_i.
\end{eqnarray}
Note that by the definition of $\mathbb{L}_{i,j}$ given by \eqref{eq0075} and rotations, $\mathfrak{g}_{i,f,d-1,j}$ is the unique solution of \eqref{eq0009} in $\bbr^{d-1}$.  Thus, by \eqref{eq0025}, rotations, the definitions of $\Pi_{c,i,j}$ and $\mathbb{L}_{i,j}$ given by \eqref{eq0075},
\begin{eqnarray}\label{eqnn0909}
-\Delta\mathfrak{g}_{i,f,d-1,j}-(1+2\log(\sum_{j=1}^{\nu}\mathfrak{g}_{j,f}))\mathfrak{g}_{i,f,d-1,j}
\gtrsim \eta^2\mathfrak{g}_{i,f,d-1,j}
\end{eqnarray}
in $D_{\sigma\eta+2,i,j}\cap\Omega_{i,j,d-1}$.  Similarly,
\begin{eqnarray}\label{eqnn0910}
-\Delta\mathfrak{g}_{i,f}^{1-\sigma}-(1+2\log(\sum_{j=1}^{\nu}\mathfrak{g}_{j,f}))\mathfrak{g}_{i,f}^{1-\sigma}
\gtrsim \mathfrak{g}_{i,f}^{1-\sigma}
\end{eqnarray}
in $(\cup_{j=1}^{\nu}B_{R}(y_{j,f}))^c$.  For every $x\in\Omega_{i,j,d-1}$, by the fact that $\Pi_{c,i,j}\perp \mathbb{L}_{i,j}$ for all $c\in\bbr$, we can re-write $x=(\alpha_{c,i,j}, z^{i,j})$, where $\alpha_{c,i,j}>-\frac12$ is given by \eqref{eqnn0003} and $z^{i,j}\perp \mathbb{L}_{i,j}$.  Now, let
\begin{eqnarray*}
\phi(\alpha_{c,i,j}, z^{i,j})=\mathfrak{g}_{i,f}^{1-\sigma}\varphi(\alpha_{c,i,j})+\eta^{-2}e^{-\frac{1}{8}(\eta_{i,j}^2-\eta^2)}\mathfrak{g}_{i,f,d-1,j}(1-\varphi(\alpha_{c,i,j}))
\end{eqnarray*}
where $\varphi(\alpha_{c,i,j})$ is the unique solution of the following equation:
\begin{eqnarray*}
\left\{\aligned&-\varphi''-\varphi'+\varphi=1,\quad\text{in }(\frac{\sigma\eta-1}{\eta_{i,j}}-\frac{1}{2}, \frac{\sigma\eta}{\eta_{i,j}}-\frac{1}{2}),\\
&\varphi'(\frac{\sigma\eta-1}{\eta_{i,j}}-\frac{1}{2})=\varphi'(\frac{\sigma\eta}{\eta_{i,j}}-\frac{1}{2})=0.
\endaligned\right.
\end{eqnarray*}
Then by \eqref{eqnn0909} and \eqref{eqnn0910}, in $(\cup_{j=1}^{\nu}B_{R}(y_{j,f}))^c$,
\begin{eqnarray*}
-\Delta\phi-(1+2\log(\sum_{j=1}^{\nu}\mathfrak{g}_{j,f}))\phi\gtrsim\left\{\aligned&e^{-\frac{1}{8}(\eta_{i,j}^2-\eta^2)}\mathfrak{g}_{i,f,d-1,j},\quad\alpha_{c,i,j}\leq\frac{\sigma\eta-1}{\eta_{i,j}}-\frac{1}{2},\\
&\mathfrak{g}_{i,f}^{1-\sigma},\quad\alpha_{c,i,j}\geq\frac{\sigma\eta}{\eta_{i,j}}-\frac{1}{2},
\endaligned\right.
\end{eqnarray*}
which, together with the maximum principle, implies that
\begin{eqnarray}\label{eq0023}
|\psi|\lesssim (\|h\|_{\natural}+\|\psi\|_{L^{\infty}(\partial (\cup_{j=1}^{\nu}B_{R}(y_{j,f}))^c})\phi(\alpha_{c,i,j}, z^{i,j})
\end{eqnarray}
in $(\cup_{j=1}^{\nu}B_{R}(y_{j,f}))^c$ for $R>0$ sufficiently large.
Based on the a-priori estimate~\eqref{eq0023}, we shall prove the a-priori estimate
\begin{eqnarray}\label{eq0026}
\|\psi\|_{\sharp}\lesssim\|h\|_{\natural}\quad\text{uniformly as }\delta\to0.
\end{eqnarray}
Since the proofs for \eqref{eq0026}, based on the blow-up arguments, are standard nowadays (cf. \cite{LNW07,MPW12,WW22}), we only sketch it here.  We assume the contrary that there exists $\delta_n\to0$, $\{\psi_n\}$ solves \eqref{eqnn0018} with $\{h_n\}\subset L^2(\bbr^d)$ satisfying $\|\psi_n\|_{\sharp}=1$ and $\|h_n\|_{\natural}=o_n(1)$ as $n\to\infty$.  Since $\delta_n\to0$, by \eqref{eqnn0030},
\begin{eqnarray*}
\log(\sum_{j=1}^{\nu}\mathfrak{g}_{j,f}(x+y_{i,f}))\to\log\mathfrak{g}=-\frac{1}{2}|x|^2+\frac{d+1}{2}
\end{eqnarray*}
in $\bbr^d$ as $n\to\infty$.  Now, let
\begin{eqnarray*}
\psi_{i,n}(x)=\phi_n(x+y_{i,f_n}),
\end{eqnarray*}
then by $\delta_n\to0$ as $n\to\infty$ and \eqref{eq0028}, it is standard to prove that $\psi_{i,n}\to \overline{\psi}_{i}$ uniformly in every compact set of $\bbr^d$ for every $i$, where $\overline{\psi}_{i}$ are bounded solutions of the following equation
\begin{eqnarray*}
-\Delta\overline{\phi}+(|x|^2-d-2)\overline{\phi}=0\quad\text{in }\bbr^d.
\end{eqnarray*}
By \cite[Theorem~1.3]{dAMS14} (see also \cite[Theorem~7.5]{GS11}), $\overline{\psi}_i=\sum_{j=1}^d a_{j}\partial_{x_j}\mathfrak{g}$.  On the other hand, by \eqref{eq0002} and \eqref{eq0011}, we can pass to the limit in the orthogonal conditions of $\psi_{i,n}$ given by \eqref{eq0029}, which implies that
\begin{eqnarray*}
0&=&\lim_{n\to+\infty}\langle \psi_{i,n}, \partial_{x_j}\mathfrak{g}\rangle_{H^1(\bbr^d)}\\
&=&\lim_{n\to+\infty}\int_{\bbr^d}(2+2\log(\sum_{j=1}^{\nu}\mathfrak{g}_{j,f_n}))\psi_n\partial_{x_j}\mathfrak{g}_{i,f_n}dx\\
&=&\int_{\bbr^d}(3+d-|x|^2)\overline{\psi}_{i}\partial_{x_j}\mathfrak{g}dx\\
&=&\langle \overline{\psi}_{i}, \partial_{x_j}\mathfrak{g}\rangle_{H^1(\bbr^d)}\\
\end{eqnarray*}
for all $j=1,2,\cdots,d$.  Since
\begin{eqnarray*}
\langle \partial_{x_i}\mathfrak{g}, \partial_{x_j}\mathfrak{g}\rangle_{H^1(\bbr^d)}=\int_{\bbr^d}(d+3+|x|^2)\partial_{x_i}\mathfrak{g}\partial_{x_j}\mathfrak{g}dx=0,
\end{eqnarray*}
we have $\overline{\psi}_{i}=0$, which implies that $\psi_{i,n}\to 0$ uniformly in every compact set of $\bbr^d$ for all $i$.  Thus, by \eqref{eq0023}, we have $\|\psi_n\|_{\natural}=o_n(1)$ which is a contradiction.
Thanks to the a-priori estimate~\eqref{eq0026}, by the Fredholm alternative, we know that the linear equation~\eqref{eqnn0018} is unique solvable for all $h\in \mathbb{Y}$.  The estimate $\sum_{i=1}^{\nu}\sum_{j=1}^d |b_{j,i}|\lesssim\|h\|_{\natural}$ comes from \eqref{eqnn0010} and \eqref{eq0028}.
\end{proof}

\vskip0.12in

By direct calculations,
\begin{eqnarray}\label{eqnnnn0001}
\max_{\alpha_{c,i,j}>-\frac12}e^{-(\frac{\alpha_{c,i,j}^2}{2}+\alpha_{c,i,j}+\frac12)\eta_{i,j}^2}\log(1+e^{(\alpha_{c,i,j}+\frac12)\eta_{i,j}^2})\sim e^{-\frac{1}{8}\eta_{i,j}^2}.
\end{eqnarray}
Thus, by taking $\sigma>0$ sufficiently small and Lemma~\ref{lem0001}, we have $\|E\|_{\natural}\lesssim e^{-\frac{1}{8}\eta^2}$.  We define
\begin{eqnarray}\label{eqnn0100}
\mathbb{B}=\{\phi\in\mathbb{X}^{\perp}\mid \|\phi\|_{\sharp}\leq M e^{-\frac{1}{8}\eta^2}\}
\end{eqnarray}
where $M>0$ is a sufficiently large constant.
\begin{lemma}\label{lem0005}
There exists $M>0$ sufficiently large such that \eqref{eqnn0034} has a unique solution $\phi_f\in \mathbb{B}$ with $\|\phi_f\|_{\sharp}+\sum_{j=1}^{d}\sum_{i=1}^{\nu}|a_{j,i}|\lesssim e^{-\frac{1}{8}\eta^2}$ as $\|f\|_{H^{-1}}\to0$.  Moreover, $\|\phi_f\|_{H^1(\bbr^d)}\lesssim e^{-\frac{1}{8}\eta^2}$ as $\|f\|_{H^{-1}}\to0$.
\end{lemma}
\begin{proof}
The proof is standard nowadays, so we also sketch it here.  For $\phi\in \mathbb{B}$, by \eqref{eq0031} and the Taylor expansion,
\begin{eqnarray}\label{eq0080}
N(\phi_+)=2\log(1+\frac{\theta \phi_-}{\sum_{j=1}^{\nu}\mathfrak{g}_{j,f}})\phi_+\text{ and }N(\phi_-)=-2\log(1-\frac{\theta \phi_-}{\sum_{j=1}^{\nu}\mathfrak{g}_{j,f}})\phi_-,
\end{eqnarray}
in $\Omega_{i}$, where $\Omega_i$ is given by \eqref{eq0045} and $\phi_{\pm}=\max\{\pm\phi, 0\}$.  Note that for $\phi\in \mathbb{B}$,
\begin{eqnarray}
\frac{\pm\phi_\pm}{\sum_{j=1}^{\nu}\mathfrak{g}_{j,f}}
&\lesssim&\|\phi\|_{\sharp}\sum_{i=1}^{\nu}\sum_{j\not=i}(\sup_{D_{\sigma\eta-1,i,j}^c\cap\Omega_i}\mathfrak{g}_{i,f}^{-\sigma}
+\sup_{D_{\sigma\eta-1,i,j}\cap\Omega_{i,j,d-1}}\frac{e^{-\frac{1}{8}(\eta_{i,j}^2-\eta^2)}\mathfrak{g}_{i,f,d-1,j}}{\eta^{2}\mathfrak{g}_{i,f}})\notag\\
&\lesssim&\eta^{-2}.\label{eqnnn0080}
\end{eqnarray}
Thus, by \eqref{eq0080} and the symmetry of $\mathfrak{g}_{i,f}$,
\begin{eqnarray*}
\|N(\phi_\pm)\|_{\natural}&\lesssim&\eta^{-2}\sum_{i=1}^{\nu}\sum_{j\not=i}(\sup_{D_{\sigma\eta,i,j}^c\cap\Omega_i}\frac{|\phi_\pm|}{\mathfrak{g}_{i,f}^{1-\sigma}}
+\sup_{D_{\sigma\eta,i,j}\cap\Omega_{i,j,d-1}}\frac{|\phi_\pm|}{e^{-\frac{1}{8}(\eta_{i,j}^2-\eta^2)}\mathfrak{g}_{i,f,d-1,j}})\notag\\
&\lesssim&\eta^{-2}e^{-\frac{1}{8}\eta^2},
\end{eqnarray*}
which implies that $\|N(\phi)\|_{\natural}\lesssim \eta^{-2}e^{-\frac{1}{8}\eta^2}$.  Now, we can solve \eqref{eq0070} by the standard fix-point arguments in $\mathbb{B}$ by choosing a sufficiently large $M>0$.  The estimate $\|\phi_f\|_{\sharp}+\sum_{j=1}^{d}\sum_{i=1}^{\nu}|a_{j,i}|\lesssim e^{-\frac{1}{8}\eta^2}$ comes from Lemma~\ref{lem0004}.
By multiplying \eqref{eqnn0034} with $\phi_f$ on both sides and integrating by parts and using the fact that $\phi_f\in\mathbb{B}$, we have $\|\phi_f\|_{H^1(\bbr^d)}^2\lesssim e^{-\frac{1}{4}\eta^2}$ as $\|f\|_{H^{-1}}\to0$
which completes the proof.
\end{proof}

\vskip0.12in

We recall that by \eqref{eq0034} and \eqref{eqnn0034}, the remaining term $\varphi_f=\rho_f-\phi_f$ satisfies
\begin{eqnarray}\label{eqnnn0034}
\left\{\aligned&\mathcal{L}_f(\varphi_f)=N(\phi_f+\varphi_f)-N(\phi_{f})+\sum_{j=1}^d\sum_{i=1}^{\nu}a_{j,i}\partial_{x_j}\mathfrak{g}_{i,f}+f,\quad\text{in }H^{-1},\\
&\langle \varphi_f, \partial_{x_j}\mathfrak{g}_{i,f}\rangle_{H^1(\bbr^d)}=0,\quad j=1,2,\cdots,d\quad\text{and}\quad i=1,2,\cdots,\nu.
\endaligned\right.
\end{eqnarray}
Moreover, it is well known (cf. \cite[Theorem~7.5 and Remark~7.7]{GS11}) that the eigenfunctions of the linear operator $\mathcal{L}$ given by \eqref{eq0043} forms an orthogonal basis in $L^2(\bbr^d)$, where the first eigenvalue is $-2$ with eigenspace $span\{\mathfrak{g}\}$ and the second eigenvalue is $0$ with eigenspace $span\{\partial_{x_j}\mathfrak{g}\}$.  Thus, we can write
\begin{eqnarray}\label{eq0050}
\varphi_f=\sum_{j=1}^{\nu}(c_j\mathfrak{g}_{j,f}+\sum_{l=1}^db_{l,j}\partial_{x_l}\mathfrak{g}_{j,f})+\varphi_f^{\perp},
\end{eqnarray}
where $\varphi_{f}^{\perp}$ is orthogonal to $span_{j,l}\{\mathfrak{g}_{j,f}, \partial_{x_l}\mathfrak{g}_{j,f}\}$ in $L^2(\bbr^d)$.
\begin{lemma}\label{lem0002}
As $\|f\|_{H^{-1}}\to0$, we have
\begin{eqnarray}\label{eq0069}
\|\varphi_f\|_{L^2(\bbr^d)}^2\sim\sum_{j=1}^\nu|c_j|^2+\|\varphi_f^{\perp}\|_{L^2(\bbr^d)}^2.
\end{eqnarray}
\end{lemma}
\begin{proof}
Note that we have
\begin{eqnarray}\label{eq0062}
\langle \mathfrak{g}_{j,f}, \nabla\mathfrak{g}_{j,f}\rangle_{H^1}=-\frac{1}{2}\nabla_{y_{j,f}}\|\mathfrak{g}_{j,f}\|_{H^1(\bbr^d)}^2=0
\end{eqnarray}
for all $j=1,2,\cdots,\nu$.  Thus, by \eqref{eq0044}, \eqref{eq0050} and \eqref{eq0062},
\begin{eqnarray}\label{eq0065}
0=\sum_{i\not=j}(c_i\langle \mathfrak{g}_{i,f}, \partial_{x_l}\mathfrak{g}_{j,f}\rangle_{H^1}+\sum_{m=1}^{d}b_{i,m}\langle \partial_{x_m}\mathfrak{g}_{i,f}, \partial_{x_l}\mathfrak{g}_{j,f}\rangle_{H^1})+b_{l,j}
\end{eqnarray}
for all $l=1,2,\cdots,d$ and $j=1,2,\cdots,\nu$.  Since
\begin{eqnarray*}
|\partial_{x_l}\mathfrak{g}_{i,f}|\lesssim r_{i,f}\mathfrak{g}_{i,f}\lesssim\mathfrak{g}_{i,f}^{1-\sigma}
\end{eqnarray*}
for all $l$, where $\sigma>0$ can be taken arbitrary small if necessary and $r_{i,f}=|x-y_{i,f}|$, by \cite[Lemma~3.7]{ACR07},
\begin{eqnarray}
|\langle \mathfrak{g}_{i,f}, \partial_{x_l}\mathfrak{g}_{j,f}\rangle_{H^1}|+|\langle \partial_{x_m}\mathfrak{g}_{i,f}, \partial_{x_l}\mathfrak{g}_{j,f}\rangle_{H^1}|
&\lesssim&\int_{\bbr^d}\mathfrak{g}_{i,f}^{1-\sigma'}\mathfrak{g}_{j,f}^{1-\sigma}\notag\\
&\lesssim&e^{-\frac{1-\sigma''}{2}\eta^2}\label{eq0064}
\end{eqnarray}
for all $i,j,l,m$,
which, together with \eqref{eq0065}, implies that
\begin{eqnarray}\label{eq0066}
\sum_{j=1}^{\nu}\sum_{l=1}^d|b_{l,j}|\lesssim e^{-\frac{1-\sigma''}{2}\eta^2}\sum_{j=1}^{\nu}|c_j|.
\end{eqnarray}
Here, $\sigma', \sigma''>0$ are constants which can be taken arbitrary small.  \eqref{eq0069} then follows from \eqref{eq0050} and \eqref{eq0066}.
\end{proof}

\vskip0.12in

We denote $N(\phi_f+\varphi_f)-N(\phi_{f})$ by $N_{\phi_{f}}(\varphi_f)$ for the sake of simplicity.  Then by Lemma~\ref{lem0005} and by multiplying \eqref{eqnnn0034} with $\varphi_f$ and integrating by parts, we have
\begin{eqnarray}\label{eq0035}
\langle\mathcal{L}_f(\varphi_f), \varphi_f\rangle_{L^2}\lesssim\int_{\bbr^d}N_{\phi_{f}}(\varphi_f)\varphi_fdx+(\|f\|_{H^{-1}}+e^{-\frac{1}{8}\eta^2})\|\varphi_f\|_{H^1(\bbr^d)}.
\end{eqnarray}
\begin{lemma}\label{lem0006}
As $\|f\|_{H^{-1}}\to0$, we have
\begin{eqnarray*}
\|\varphi_f\|_{H^1(\bbr^d)}\lesssim\|f\|_{H^{-1}}+e^{-(\frac{3}{8}-\sigma)\eta^2}.
\end{eqnarray*}
\end{lemma}
\begin{proof}
For the sake of clarity, we divide the proof into several parts.

{\bf Step.~1}\quad The estimate of $N_{\phi_{f}}(\varphi_f)\varphi_f$.

By the Taylor expansion and $u=\sum_{j=1}^{\nu}\mathfrak{g}_{j,f}+\phi_f+\varphi_f\geq0$,
\begin{eqnarray}
N_{\phi_{f}}(\varphi_f)&=&2(\sum_{j=1}^{\nu}\mathfrak{g}_{j,f}+\phi_f+\varphi_f)\log(\sum_{j=1}^{\nu}\mathfrak{g}_{j,f}+\phi_f+\varphi_f)\notag\\
&&-2(\sum_{j=1}^{\nu}\mathfrak{g}_{j,f}+\phi_{f})\log(\sum_{j=1}^{\nu}\mathfrak{g}_{j,f}+\phi_{f})\notag\\
&&-2(1+\log(\sum_{j=1}^{\nu}\mathfrak{g}_{j,f}))\varphi_f\label{eqnnn0036}\\
&=&2\log(1+\frac{\phi_{f}+\theta\varphi_f}{\sum_{j=1}^{\nu}\mathfrak{g}_{j,f}}))\varphi_f\label{eqnn0036}\\
&=&2\log(1+\frac{\phi_{f}}{\sum_{j=1}^{\nu}\mathfrak{g}_{j,f}}))\varphi_f+\frac{\varphi_f^2}{\sum_{i=1}^\nu\mathfrak{g}_{i,f}+\phi_{f}+\theta'\varphi_f},\label{eq0036}
\end{eqnarray}
where $\theta,\theta'\in(0, 1)$.  By \eqref{eqnnn0080}, we have
\begin{eqnarray}\label{eqnnnn0080}
|\frac{\phi_f}{\sum_{j=1}^{\nu}\mathfrak{g}_{j,f}}|\lesssim\eta^{-2}.
\end{eqnarray}
Thus,
\begin{eqnarray*}
\sum_{i=1}^\nu\mathfrak{g}_{i,f}+\phi_{f}=(1+O(\eta^{-2}))\sum_{i=1}^\nu\mathfrak{g}_{i,f}.
\end{eqnarray*}
For every $R>0$, let
\begin{eqnarray*}
\Upsilon_{R,\alpha}=\{x\in \partial(\cup_{j=1}^{\nu}B_{R}(y_{j,f}))\mid \varphi_f=\alpha\sum_{i=1}^\nu\mathfrak{g}_{i,f}\}
\end{eqnarray*}
where $\alpha\geq-1+O(\eta^{-2})$.  Since $0\leq\alpha\leq e^{\frac{R^2-(d+5)}{2}}-1+O(\eta^{-2})$ implies that
\begin{eqnarray*}
2(1+\log(1+O(\eta^{-2})+\theta\alpha))\leq R^2-(d+3)
\end{eqnarray*}
and $-1+O(\eta^{-2})+e^{-\frac{R^2-(d+1)}{2}}\leq\alpha<0$ implies that
\begin{eqnarray*}
-2(1+\log(1+O(\eta^{-2})+\theta\alpha))\leq R^2-(d+3)
\end{eqnarray*}
by $\theta\in(0, 1)$, by \eqref{eq0011}, \eqref{eqnn0036} and \eqref{eqnnnn0080}, for every $R>0$ and every $x\in\partial(\cup_{j=1}^{\nu}B_{R}(y_{j,f}))$, one of the following cases must happen:
\begin{enumerate}
\item[$(a)$]\quad $N_{\phi_{f}}(\varphi_f)\varphi_f\leq (R^2-(d+3))\varphi_f^2$,
\item[$(b)$]\quad $\varphi_f\gtrsim1$,
\item[$(c)$]\quad $\varphi_f\sim-\sum_{i=1}^\nu\mathfrak{g}_{i,f}$.
\end{enumerate}
On the other hand, by \eqref{eqnnn0036}, \eqref{eq0036} and \eqref{eqnnnn0080}, in $\Upsilon_{R,\alpha}$ with $-1\leq\alpha<0$,
\begin{eqnarray*}
2((1+O(\eta^{-2})+\alpha)\log(1+O(\eta^{-2})+\alpha)-\alpha)=O(\eta^{-2})\alpha+\frac{\alpha^2}{1+O(\eta^{-2})+\theta' \alpha}.
\end{eqnarray*}
It follows that $\theta'\to\frac12+O(\eta^{-\sigma})$ as $\alpha\to-1$.
Thus, by \eqref{eq0011},
\begin{eqnarray}
N_{\phi_{f}}(\varphi_f)\varphi_f\leq(\max\{-(2+2\log(\sum_{j=1}^{\nu}\mathfrak{g}_{j,f})), 0\}+O(\eta^{-\sigma}))|\varphi_f|^2+O(\varphi_f^3).\label{eqnn0005}
\end{eqnarray}

{\bf Step.~2}\quad The estimate of $\|\varphi_f^{\perp}\|_{H^1(\bbr^d)}^2$.

By \eqref{eq0050}, \eqref{eq0066}, \eqref{eq0035}, \eqref{eqnn0005} and Lemma~\ref{lem0002},
\begin{eqnarray*}
\langle\mathcal{L}_f(\varphi_f), \varphi_f\rangle_{L^2}&\lesssim&\int_{\bbr^d}N_{\phi_{f}}(\varphi_f)\varphi_{f}dx+(\|f\|_{H^{-1}}+e^{-\frac{1-\sigma''}{2}\eta^2}\sum_{j=1}^{\nu}|c_j|)\|\varphi_f\|_{H^1(\bbr^d)}\notag\\
&\leq&\int_{\bbr^d}\max\{-(2+2\log(\sum_{j=1}^{\nu}\mathfrak{g}_{j,f})), 0\}\varphi_f^2dx\notag\\
&&+O(\|\varphi_f\|^3_{H^1(\bbr^d)})+(\|f\|_{H^{-1}}+e^{-\frac{1-\sigma''}{2}\eta^2}\sum_{j=1}^{\nu}|c_j|)\|\varphi_f\|_{H^1(\bbr^d)}.
\end{eqnarray*}
It follows that
\begin{eqnarray}
&&\int_{\cup_{j=1}^{\nu}B_{R}(y_{j,f})}(|\nabla \varphi_f|^2-(1+2\log(\sum_{j=1}^{\nu}\mathfrak{g}_{j,f})))|\varphi_f|^2dx\notag\\
&&+\int_{\bbr^d\backslash(\cup_{j=1}^{\nu}B_{R}(y_{j,f}))}|\nabla \varphi_f|^2+|\varphi_f|^2dx\notag\\
&\leq&O(\|\varphi_f\|^3_{H^1(\bbr^d)})+(\|f\|_{H^{-1}}+e^{-\frac{1-\sigma''}{2}\eta^2}\sum_{j=1}^{\nu}|c_j|)\|\varphi_f\|_{H^1(\bbr^d)}.\label{eqnn0035}
\end{eqnarray}
for a sufficiently large $R>0$.  For every $j$,
\begin{eqnarray}
&&\int_{B_{R}(y_{j,f})}(|\nabla \varphi_f|^2-(1+2\log(\sum_{j=1}^{\nu}\mathfrak{g}_{j,f})))|\varphi_f|^2dx\notag\\
&\geq&\int_{B_{R}(y_{j,f})}[(1-\sigma)(d+1)-r_{i,f}^2-2\log(\sum_{j=1}^\nu\mathfrak{g}_{j,f}))]|\varphi_f|^2dx)\notag\\
&&+\int_{B_{R}(y_{j,f})}|\nabla \varphi_f|^2+(r_{i,f}^2-d-2)|\varphi_f|^2dx\label{eq0046}
\end{eqnarray}
where $\sigma>0$ is sufficient small.  By \cite[Theorem~7.5 and Remark~7.7]{GS11}, it is easy to show that
\begin{eqnarray}\label{eqnnn0007}
\int_{B_{R}(y_{j,f})}|\nabla \widetilde{\varphi}_{f,R}^{\perp}|^2+(r_{j,f}^2-d-2)|\widetilde{\varphi}_{f,R}^{\perp}|^2dx\gtrsim\int_{B_{R}(y_{j,f})}|\nabla \widetilde{\varphi}_{f,R}^{\perp}|^2+|\widetilde{\varphi}_{f,R}^{\perp}|^2dx
\end{eqnarray}
where $R>0$ sufficiently large and $\widetilde{\varphi}_{f,R}^{\perp}=\varphi_f^{\perp}\psi_R$ with $\psi_R$ being a smooth cut-off function such that $\psi_R=1$ in $B_{\frac{R}{4}}(0)$ and $\psi_R=0$ in $B_{R}(0)\backslash B_{\frac{R}{2}}(0)$.  Note that
\begin{eqnarray*}
&&\int_{B_{R}(y_{j,f})}|\nabla \widetilde{\varphi}_{f,R}^{\perp}|^2+(r_{j,f}^2-d-2)|\widetilde{\varphi}_{f,R}^{\perp}|^2dx\\
&\lesssim&\int_{B_{R}(y_{j,f})}|\nabla \varphi_f^{\perp}|^2+(r_{i,f}^2-d-2)|\varphi_f^{\perp}|^2dx\\
&&+\int_{B_{R}(y_{j,f})\backslash B_{\frac{R}{2}}(y_{j,f})}|\nabla \widetilde{\varphi}_{f,R}^{\perp}|^2+|\widetilde{\varphi}_{f,R}^{\perp}|^2dx\\
&\lesssim&\int_{B_{R}(y_{j,f})}|\nabla \varphi_f^{\perp}|^2+(r_{i,f}^2-d-2)|\varphi_f^{\perp}|^2dx
\end{eqnarray*}
and
\begin{eqnarray*}
\int_{B_{R}(y_{j,f})}|\nabla \widetilde{\varphi}_{f,R}^{\perp}|^2+|\widetilde{\varphi}_{f,R}^{\perp}|^2dx&=&\int_{B_{R}(y_{j,f})}|\nabla \varphi_f^{\perp}|^2+|\varphi_f^{\perp}|^2dx\\
&&+O(1)\int_{B_{R}(y_{j,f})}|\nabla \varphi_f^{\perp}|^2+(r_{i,f}^2-d-2)|\varphi_f^{\perp}|^2dx,
\end{eqnarray*}
thus, by \eqref{eqnnn0007},
\begin{eqnarray}\label{eqnn0006}
\int_{B_{R}(y_{j,f})}|\nabla \varphi_f^{\perp}|^2+(r_{i,f}^2-d-2)|\varphi_f^{\perp}|^2dx\gtrsim\int_{B_{R}(y_{j,f})}|\nabla \varphi_f^{\perp}|^2+|\varphi_f^{\perp}|^2dx.
\end{eqnarray}
On the other hand, we have $|x-y_{i,f}|\geq\eta+O(1)$ for all $j\not=i$ in $B_{R}(y_{j,f})$ which implies that
\begin{eqnarray*}
\sum_{i\not=j}\mathfrak{g}_{i,f}\lesssim e^{-\frac{\eta^2}{2}}\mathfrak{g}_{j,f}\quad\text{in }B_{R}(y_{j,f}).
\end{eqnarray*}
It follows that
\begin{eqnarray*}
-2\log(\sum_{j=1}^{\nu}\mathfrak{g}_{j,f})=-(d+1)+r_{j,f}^2+O(e^{-\frac{\eta^2}{2}})\quad\text{in }B_{R}(y_{j,f}),
\end{eqnarray*}
which implies that
\begin{eqnarray*}
|\int_{B_{R}(y_{j,f})}[(1-\sigma)(d+1)-r_{i,f}^2-2\log(\sum_{j=1}^\nu\mathfrak{g}_{j,f}))]|\varphi_f|^2dx|\lesssim\sigma\|\varphi_f\|_{H^1(\bbr^d)}^2.
\end{eqnarray*}
It follows from \eqref{eqnn0035}, \eqref{eq0046}, \eqref{eqnn0006} and Lemma~\ref{lem0002} that
\begin{eqnarray}
\|\varphi_f^{\perp}\|_{H^1(\bbr^d)}^2\lesssim\|f\|_{H^{-1}}^2+e^{-(1-\sigma'')\eta^2}+\sum_{j=1}^{\nu}|c_j|^2.\label{eqnnn0035}
\end{eqnarray}

{\bf Step.~3}\quad The estimate of $\sum_{j=1}^{\nu}|c_j|^2$.

Let
\begin{eqnarray*}
\widetilde{\psi}_{j}=\mathfrak{g}_{j,f}\widehat{\psi}_j,
\end{eqnarray*}
where $\widehat{\psi}_j$ is a smooth cut-off function satisfying
\begin{eqnarray*}
\widehat{\psi}_j=\left\{\aligned&1,\quad|x-y_{j,f}|\leq(\frac12-\sigma)\eta,\\
&0,\quad |x-y_{j,f}|\geq(\frac12-\sigma)\eta+1.
\endaligned
\right.
\end{eqnarray*}
Note that $\mathfrak{g}_{j,f}$ is the unique positive solution of \eqref{eq0009}, thus, by multiplying \eqref{eqnnn0034} with $-sgn(c_j)\widetilde{\psi}_{j}$ and integrating by parts, we have
\begin{eqnarray}
&&-sgn(c_j)\int_{\bbr^d}(N_{\phi_{f}}(\varphi_f)+\sum_{j=1}^d\sum_{i=1}^{\nu}a_{j,i}\partial_{x_j}\mathfrak{g}_{i,f})\widetilde{\psi}_j+f\widetilde{\psi}_jdx\notag\\
&=&-sgn(c_j)\int_{\bbr^d}-\Delta\mathfrak{g}_{j,f}\widehat{\psi}_j\varphi_{f}-2\varphi_f\nabla\widehat{\psi}_j\nabla\mathfrak{g}_{j,f}
-\Delta\widehat{\psi}_j\mathfrak{g}_{j,f}\varphi_fdx\notag\\
&&+sgn(c_j)\int_{\bbr^d}(1+2\log(\sum_{j=1}^\nu\mathfrak{g}_{j,f}))\widehat{\psi}_j\mathfrak{g}_{j,f}\phi_fdx\notag\\
&=&sgn(c_j)\int_{\bbr^d}(r_{j,f}^2+1-d+2\log(\sum_{j=1}^\nu\mathfrak{g}_{j,f}))\widehat{\psi}_j\mathfrak{g}_{j,f}\phi_fdx\notag\\
&&+sgn(c_j)\int_{\bbr^d}2\phi_f\nabla\widehat{\psi}_j\nabla\mathfrak{g}_{j,f}
+\Delta\widehat{\psi}_j\mathfrak{g}_{j,f}\phi_fdx.\label{eq0051}
\end{eqnarray}
By \eqref{eq0064} and similar estimates of \eqref{eqnn0005},
\begin{eqnarray*}
\int_{\bbr^d}N_{\phi_{f}}(\varphi_f)\widetilde{\psi}_jdx&\leq& |c_j|\int_{\bbr^d}\max\{-(2+2\log(\sum_{j=1}^{\nu}\mathfrak{g}_{j,f})), 0\}|\mathfrak{g}_{j,f}|^2\widehat{\psi}_jdx\\
&&+o(\sum_{i=1}^{\nu}|c_i|+\|\phi_f^{\perp}\|_{H^1(\bbr^d)})+\sum_{i=1}^{\nu}|c_i|^2+\|\phi_f^{\perp}\|_{H^1(\bbr^d)}^2.
\end{eqnarray*}
Note that supp$(\widehat{\psi}_j)\subset\Omega_j$ with $\sum_{i\not=j}\mathfrak{g}_{i,f}\lesssim e^{-\sigma\eta^2}\mathfrak{g}_{j,f}$ in supp$(\widehat{\psi}_j)$, thus, by \eqref{eq0011}, \eqref{eq0050} and Lemma~\ref{lem0002},
\begin{eqnarray*}
\int_{\bbr^d}2\phi_f\nabla\widehat{\psi}_j\nabla\mathfrak{g}_{j,f}
+\Delta\widehat{\psi}_j\mathfrak{g}_{j,f}\phi_fdx=o(\sum_{i=1}^{\nu}|c_i|+\|\phi_f^{\perp}\|_{H^1(\bbr^d)})
\end{eqnarray*}
and
\begin{eqnarray*}
&&sgn(c_j)\int_{\bbr^d}(r_{j,f}^2+1-d+2\log(\sum_{j=1}^\nu\mathfrak{g}_{j,f}))\widehat{\psi}_j\mathfrak{g}_{j,f}\phi_fdx-|\int_{\bbr^d}N_{\phi_{f}}(\varphi_f)\widetilde{\psi}_jdx|\\
&\sim& |c_j|+o(\sum_{i=1}^{\nu}|c_i|+\|\phi_f^{\perp}\|_{H^1(\bbr^d)}),
\end{eqnarray*}
which, together with \eqref{eq0051} and Lemma~\ref{lem0005}, implies that
\begin{eqnarray}\label{eq0052}
|c_j|+o(\sum_{i\not=j}c_i+\|\phi_f^{\perp}\|_{H^1(\bbr^d)})\lesssim O(e^{-(\frac{3}{8}-\sigma)\eta^2})+\|f\|_{H^{-1}}.
\end{eqnarray}
Since \eqref{eq0052} holds for all $j=1,2,\cdots,\nu$ and the linear part of this inequality is diagonally dominant, we have
\begin{eqnarray}\label{eq0056}
\sum_{j=1}^{\nu}|c_j|^2\lesssim O(e^{-(\frac{3}{4}-2\sigma)\eta^2})+\|f\|_{H^{-1}}^2.
\end{eqnarray}

{\bf Step.~4}\quad The estimate of $\|\varphi_f\|_{H^1(\bbr^d)}$.

\eqref{eq0069} comes from \eqref{eqnnn0035}, \eqref{eq0056} and Lemma~\ref{lem0002}.
\end{proof}

\vskip0.12in

We also need the following lemma.
\begin{lemma}\label{lem0007}
As $\|f\|_{H^{-1}}\to0$, we have $e^{-\frac{1}{8}\eta^2}\lesssim\|f\|_{H^{-1}}$.
\end{lemma}
\begin{proof}
Without loss of generality, we may assume that $\eta=|y_{1,f}-y_{2,f}|$.  We define
\begin{eqnarray}\label{eq0089}
\widehat{\psi}=\mathfrak{g}_{1,f}\overline{\psi},
\end{eqnarray}
where $\overline{\psi}$ is a smooth cut-off function satisfying
\begin{eqnarray*}
\overline{\psi}=\left\{\aligned&1,\quad x\in B_1(\frac{1}{2}(y_{1,f}+y_{2,f})+\frac{2\log\eta}{\eta^2}(y_{2,f}-y_{1,f})+5),\\
&0,\quad B_{2}^c(\frac{1}{2}(y_{1,f}+y_{2,f})+\frac{2\log\eta}{\eta^2}(y_{2,f}-y_{1,f})+5).
\endaligned
\right.
\end{eqnarray*}
By similar estimates of \eqref{eq0036}, we know that $N(\rho_f)\geq0$.  Thus,
similar to \eqref{eq0051}, by multiplying \eqref{eq0034} with $-\widehat{\psi}$ and integrating by parts and \eqref{eq0036}, we have
\begin{eqnarray}
\|\widehat{\psi}\|_{H^1(\bbr^d)}\|f\|_{H^{-1}}&\gtrsim&\int_{\bbr^d}(r_{1,f}^2+1-d+2\log(\sum_{j=1}^\nu\mathfrak{g}_{j,f}))\overline{\psi}\mathfrak{g}_{1,f}\rho_fdx\notag\\
&&+\int_{\bbr^d}(E+N(\rho_f))\widehat{\psi}dx+\int_{\bbr^d}2\rho_f\nabla\overline{\psi}\nabla\mathfrak{g}_{1,f}
+\Delta\overline{\psi}\mathfrak{g}_{1,f}\rho_fdx\notag\\
&\gtrsim&\int_{\bbr^d}(r_{1,f}^2+1-d+2\log(\sum_{j=1}^\nu\mathfrak{g}_{j,f}))\overline{\psi}\mathfrak{g}_{1,f}\rho_fdx+\int_{\bbr^d}E\overline{\psi}dx\notag\\
&&\int_{\bbr^d}2\rho_f\nabla\overline{\psi}\nabla\mathfrak{g}_{1,f}
+\Delta\overline{\psi}\mathfrak{g}_{1,f}\rho_fdx.\label{eq0157}
\end{eqnarray}
By \eqref{eq0011} and \eqref{eqnnnn0001}, $\int_{\bbr^d}E\widehat{\psi}dx\sim e^{-\frac{1}{4}\eta^2}$.
By Lemmas~\ref{lem0005} and \ref{lem0006},
\begin{eqnarray*}
&&|\int_{\bbr^d}2\rho_f\nabla\overline{\psi}\nabla\mathfrak{g}_{1,f}
+\Delta\overline{\psi}\mathfrak{g}_{1,f}\rho_fdx|\\
&&+|\int_{\bbr^d}(r_{1,f}^2+1-d+2\log(\sum_{j=1}^\nu\mathfrak{g}_{j,f}))\overline{\psi}\mathfrak{g}_{1,f}\rho_fdx|\\
&\lesssim&\eta^{-2}e^{-\frac{1}{4}\eta^2}+e^{-\sigma\eta^2}\|f\|_{H^{-1}}.
\end{eqnarray*}
Note that $\|\widehat{\varphi}\mathfrak{g}_{1,f}\|_{H^1(\bbr^d)}\lesssim e^{-\frac{1}{8}\eta^2}$,
thus, we obtain $e^{-\frac{1}{8}\eta^2}\lesssim\|f\|_{H^{-1}}$.
\end{proof}

\vskip0.12in

We close this section by the following proposition.
\begin{proposition}\label{propnn0001}
As $\|f\|_{H^{-1}}\to0$, we have $\|\rho_f\|_{H^1(\bbr^d)}\lesssim\|f\|_{H^{-1}}$.
\end{proposition}
\begin{proof}
The conclusion follows immediately from Lemmas~\ref{lem0005}, \ref{lem0006} and \ref{lem0007}.
\end{proof}

\section{Optimality of (\ref{eqnn0002})}
In this section, we shall construct an example to show that (\ref{eqnn0002}) is optimal.
Let $\mathfrak{g}_{\pm L}=\mathfrak{g}(x\pm \frac{L}{2}e_1)$ with $e_1=(1,0,\cdots,0)$ and we consider the following equation:
\begin{eqnarray}\label{eq0070}
\left\{\aligned&\mathcal{L}_L(\rho_L)=E_L+N(\rho_{L})-\sum_{j=1}^d (a_{j,+,L}\partial_{x_j}\mathfrak{g}_{L}+a_{j,-,L}\partial_{x_j}\mathfrak{g}_{-L}),\quad\text{in }H^{-1},\\
&\langle \rho_L, \partial_{x_j}\mathfrak{g}_{\pm L}\rangle_{H^1(\bbr^d)}=0,\quad j=1,2,\cdots,d,
\endaligned\right.
\end{eqnarray}
where
\begin{eqnarray*}
\mathcal{L}_L=-\Delta-1-2\log(\mathfrak{g}_{L}+\mathfrak{g}_{-L})
\end{eqnarray*}
is the linear operator,
\begin{eqnarray*}
E_L=2(\mathfrak{g}_{L}+\mathfrak{g}_{-L})\log(\mathfrak{g}_{L}+\mathfrak{g}_{-L})-2(\mathfrak{g}_{L}\log\mathfrak{g}_{L}+\mathfrak{g}_{-L}\log\mathfrak{g}_{-L})
\end{eqnarray*}
is the error,
\begin{eqnarray*}
N(\rho_{L})&=&2(\mathfrak{g}_{L}+\mathfrak{g}_{-L}+\rho_{L})\log(\mathfrak{g}_{L}+\mathfrak{g}_{-L}+\rho_{L})
-2(\mathfrak{g}_{L}+\mathfrak{g}_{-L})\log(\mathfrak{g}_{L}+\mathfrak{g}_{-L})\notag\\
&&-2(1+\log(\mathfrak{g}_{L}+\mathfrak{g}_{-L}))\rho_{L}\label{eq0073}
\end{eqnarray*}
is the nonlinear part, and $a_{j,\pm,L}$ are Lagrange multipliers given by
\begin{eqnarray*}\label{eq0074}
a_{j,\pm,L}\sim\langle E_{L}+N(\rho_{L}), \partial_{x_j}\mathfrak{g}_{\pm L}\rangle_{L^2}.
\end{eqnarray*}
By Lemma~\ref{lem0005}, \eqref{eq0070} is unique solvable in $\mathbb{B}$, where $\mathbb{B}$ is given by \eqref{eqnn0100}.  Moreover, by Lemma~\ref{lem0001},
\begin{eqnarray}\label{eq0081}
\|\rho_{L}\|_{\sharp}+\sum_{j=1}^{d}\sum_{i=1}^{\nu}|a_{j,\pm,L}|\lesssim e^{-\frac{1}{8}L^2}.
\end{eqnarray}

\vskip0.12in

Let
\begin{eqnarray*}
u_L=\mathfrak{g}_{L}+\mathfrak{g}_{-L}+\rho_{L}.
\end{eqnarray*}
Then by similar estimates for \eqref{eqnnnn0080}, $u_L\sim\mathfrak{g}_{L}+\mathfrak{g}_{-L}>0$.  Moreover, by the classical regularity, $\rho_{L}\in C^{1,\alpha}(\bbr^d)$ for all $\alpha\in(0, 1)$.  It follows from \eqref{eq0070} that
\begin{eqnarray}\label{eq0083}
f_L&:=&-\Delta u_L+u_L-2u_L\log|u_L|\notag\\
&=&-\sum_{j=1}^d (a_{j,+,L}\partial_{x_j}\mathfrak{g}_{L}+a_{j,-,L}\partial_{x_j}\mathfrak{g}_{-L}).
\end{eqnarray}
\begin{proposition}\label{propnn0002}
As $L\to+\infty$, we have
\begin{eqnarray*}
\inf_{z_i\in\bbr^d}\|u_L-\sum_{i=1}^2\mathfrak{g}(\cdot-z_i)\|_{H^1(\bbr^d)}^2\sim\|f_L\|_{H^{-1}}^2.
\end{eqnarray*}
\end{proposition}
\begin{proof}
Thanks to Lemma~\ref{lem0007} and \eqref{eq0081}, we have
\begin{eqnarray}\label{eq0085}
\|f_L\|_{H^{-1}}\sim e^{-\frac{1}{8}L^2}.
\end{eqnarray}
Now, we consider the minimizing problem
\begin{eqnarray}\label{eq0086}
\mathfrak{c}_L=\inf_{z_i\in\bbr^d}\|u_L-\sum_{i=1}^2\mathfrak{g}(\cdot-z_i)\|_{H^1(\bbr^d)}^2.
\end{eqnarray}
Clearly, by \eqref{eq0081}, $\mathfrak{c}_L\lesssim\|\rho_L\|_{H^1(\bbr^d)}^2\lesssim e^{-\frac{1}{4}L^2}$.
As before, we can also write
\begin{eqnarray}\label{eq0087}
u_L=\sum_{i=1}^2\mathfrak{g}(\cdot-y_{i,L})+\rho_{L}^*
\end{eqnarray}
where $\{y_{i,L}\}$ is the solution of \eqref{eq0086} and the remaining term $\rho_L^*$ satisfies
\begin{eqnarray}\label{eq0088}
\|\rho_L^*\|_{H^1(\bbr^d)}^2=\mathfrak{c}_L\lesssim e^{-\frac{1}{4}L^2}.
\end{eqnarray}
and the orthogonal conditions
\begin{eqnarray*}
\langle \rho_L^*, \partial_{x_l}\mathfrak{g}_{j,L}\rangle_{H^1(\bbr^d)}=0,\quad l=1,2,\cdots,d\text{ and }j=1,2,
\end{eqnarray*}
where $\mathfrak{g}_{j,L}=\mathfrak{g}(\cdot-y_{j,L})$.  By \eqref{eq0087} and \eqref{eq0088}, we may assume that $y_{1,L}=\frac{L}{2}e_1+o(1)$ and $y_{2,L}=-\frac{L}{2}e_1+o(1)$.  By \eqref{eq0083} and \eqref{eq0087}, $\rho_L^*$ satisfies the equation
\begin{eqnarray}\label{eq0091}
\left\{\aligned&-\Delta\rho_L^*+\rho_L^*=2u_L\log u_L-2(\sum_{j=1}^2\mathfrak{g}_{j,L}\log\mathfrak{g}_{j,L})+f_L,\quad\text{in }\bbr^d,\\
&\langle \rho_L^*, \partial_{x_j}\mathfrak{g}_{j,L}\rangle_{H^1(\bbr^d)}=0,\quad l=1,2,\cdots,d\text{ and }j=1,2.
\endaligned\right.
\end{eqnarray}
As before, we can write
\begin{eqnarray*}
2u_L\log u_L-2(\sum_{j=1}^2\mathfrak{g}_{j,L}\log\mathfrak{g}_{j,L})=E_L+N_L(\rho_L^*)+2(1+\log(\sum_{j=1}^2\mathfrak{g}_{j,L}))\rho_L^*,
\end{eqnarray*}
which, together with \eqref{eq0083}, implies that we can re-write \eqref{eq0091} as follows:
\begin{eqnarray*}\label{eqnn0091}
\left\{\aligned&-\mathcal{L}_L(\rho_L^*)=E_L+N_L(\rho_L^*)-\sum_{j=1}^d (a_{j,+,L}\partial_{x_j}\mathfrak{g}_{L}+a_{j,-,L}\partial_{x_j}\mathfrak{g}_{-L}),\quad\text{in }\bbr^d,\\
&\langle \rho_L^*, \partial_{x_j}\mathfrak{g}_{j,L}\rangle_{H^1(\bbr^d)}=0,\quad l=1,2,\cdots,d\text{ and }j=1,2.
\endaligned\right.
\end{eqnarray*}
By Lemma~\ref{lem0004} and similar estimates in the proof of Lemma~\ref{lem0005}, we have
\begin{eqnarray*}
\|\rho_L^*\|_{\sharp}\lesssim e^{-\frac{1}{8}L^2}.
\end{eqnarray*}
Now, using $\widehat{\psi}$, given by \eqref{eq0089}, as the test function of \eqref{eq0091} and estimating as that for \eqref{eq0157}, we have $\|\rho_L^*\|_{H^1(\bbr^d)}\gtrsim e^{-\frac{1}{8}L^2}$.
It follows from \eqref{eq0088} that $\|\rho_L^*\|_{H^1(\bbr^d)}\sim e^{-\frac{1}{8}L^2}$,
which, together with \eqref{eq0085} and \eqref{eq0088}, implies that $\|\rho_L^*\|_{H^1(\bbr^d)}\sim\|f_L\|_{H^{-1}}$.  It completes the proof.
\end{proof}

We close this section by the proof of Theorem~\ref{thmnn0001}.

\vskip0.12in

\noindent\textbf{Proof of Theorem~\ref{thmnn0001}:}\quad
The conclusion follows immediately from Propositions~\ref{propnn0001} and \ref{propnn0002}.
\hfill$\Box$

\end{document}